\documentclass[12pt]{amsart}

\usepackage{enumerate,url,amssymb, upref}
\usepackage{graphicx}
\usepackage{psfrag}
\newtheorem{theorem}{Theorem}[section]
\newtheorem{lemma}[theorem]{Lemma}
\newtheorem{proposition}[theorem]{Proposition}
\newtheorem{corollary}[theorem]{Corollary}

\theoremstyle{definition}
\newtheorem{definition}[theorem]{Definition}

\newtheorem{question}[theorem]{Question}

\theoremstyle{remark}
\newtheorem{remark}[theorem]{Remark}

\numberwithin{equation}{section}

\newcommand{\abs}[1]{\lvert#1\rvert}
\newcommand{\norm}[1]{\lVert#1\rVert}
\newcommand{\inn}[1]{\langle#1\rangle}
\newcommand{\C}{\mathbb{C}}
\newcommand{\HH}{\mathbb{H}}
\newcommand{\M}{\mathbb{M}}
\newcommand{\R}{\mathbb{R}}
\newcommand{\Z}{\mathbb{Z}}

\DeclareMathOperator{\diam}{diam}
\DeclareMathOperator{\dist}{dist}
\DeclareMathOperator{\sign}{sign}
\DeclareMathOperator{\re}{Re}
\DeclareMathOperator{\im}{Im}

\DeclareMathOperator{\loc}{loc}
\DeclareMathOperator*{\var}{var}
\DeclareMathOperator*{\osc}{osc}

\def\XXint#1#2#3{{\setbox0=\hbox{$#1{#2#3}{\int}$}
\vcenter{\hbox{$#2#3$}}\kern-.5\wd0}}

\def\le{\leqslant}
\def\ge{\geqslant}
\begin{document}

\title{Variation of quasiconformal mappings on lines}

\author{Leonid V. Kovalev}
\address{Department of Mathematics, Syracuse University, 215 Carnegie, Syracuse,
NY 13244, USA}
\email{lvkovale@syr.edu}
\thanks{Kovalev was supported by the NSF grant DMS-0700549.}

\author{Jani Onninen}
\address{Department of Mathematics, Syracuse University, 215 Carnegie, Syracuse,
NY 13244, USA}
\email{jkonnine@syr.edu}
\thanks{Onninen was supported by the NSF grant  DMS-0701059.}


\date{January 9, 2009}


\begin{abstract}
We obtain improved regularity of homeomorphic solutions of the reduced  Beltrami equation, as compared to the standard Beltrami equation. Such an improvement is not possible in terms of H\"older or Sobolev  regularity; instead, our results concern the generalized variation of restrictions to lines. Specifically, we prove that the restriction to any line segment has finite $p$-variation for all $p>1$ but not necessarily for $p=1$.
\end{abstract}

\maketitle

\section{Introduction}

A key property of Sobolev functions in Euclidean spaces is their absolute continuity on almost every line
parallel to the coordinate axes. The restrictions to arbitrary lines need not be even bounded
for functions in Sobolev spaces $W^{1,s}$, $1\le s\le n$. However, for $s>n$
the restriction of a Sobolev function to \textit{any} line has finite $p$-variation with $p=s/(s-n+1)$, see Remark~\ref{morrey}. Here we refer to the generalized variation~\cite{MO,Wa,Yo}, which is defined as follows.

For distinct points $a,b\in \R^n$ we write $[a,b]= \{(1-t)a + tb \colon 0 \le t \le 1\}$ and call $[a,b]$ the line segment with the endpoints $a$ and $b$.  Any partition $0=t_0<t_1<\dots<t_N=1$ induces a partition of $[a,b]$ by the rule  $a_j=a+t_j(b-a)$, $j=0,\dots,N$.

\begin{definition}\label{phivar} Let $\phi\colon [0,\infty)\to [0,\infty)$ be a convex strictly increasing function
such that $\phi(0)=0$. A mapping $f$ from a line segment  $[a,b]\subset\R^n$ into $\R^m$ has \textit{finite $\phi$-variation} on $[a,b]$ if
\[
\var_{[a, b]} (f; \phi):= \sup \sum_{j=1}^N \phi(\abs{f(a_j)-f(a_{j-1})})<\infty,
\]
where the supremum is taken over all partitions $(a_j)_{j=0}^N$ of $[a,b]$ and over all $N\ge 1$.
\end{definition}

If $f$ is a mapping from a domain $\Omega\subset\R^n$ into $\R^m$, then we say that $f$ has \textit{finite $\phi$-variation on lines}
if its restriction to any compact line segment contained in $\Omega$ has finite $\phi$-variation.
When $\phi(t)=t^p$, we speak of $p$-variation (or simply variation if $p=1$) and write $\var_{[a,b]} (f; p)$.

We are primarily interested in the variation of quasiconformal mappings $f\colon \R^n \to\R^n$, where
  $n\ge 2$. Recall that a sense-preserving homeomorphism $f \colon \R^n \to \R^n$ is said to be {\it quasiconformal} if there exists $H < \infty$ such that
\begin{equation}\label{distine1}
\limsup\limits_{r \to 0} \frac{\underset{{|x-a|=r}}{\max} \left|f(x) - f(a)\right|}{\underset{{|y-a|=r}}{\min} \left|f(y) - f(a)\right|} \leqslant H \, , \quad \textnormal{ for every } a\in \R^n .
\end{equation}
 It is a well-known result of Gehring~\cite{Ge60} that such mappings are absolutely continuous on almost every line (ACL). This was recently extended to Ahlfors regular metric spaces by Balogh, Koskela and Rogovin~\cite{BKR}.
 The ACL property makes it possible to give an analytic definition of quasiconformal mappings.
 \begin{definition}\label{qc}
 A homeomorphism  $f \in W^{1,n}_{\loc}(\R^n ; \mathbb R^n)$ is $K$-quasiconformal, $1 \leqslant   K < \infty$, if it satisfies  the  distortion inequality
\begin{equation}\label{Kav22}
\norm{Df(x)}^n \leqslant   K   J(x,f) \quad \quad \mbox{a.e.}
\end{equation}
Here $\norm{Df(x)}$ stands for the norm of the differential matrix and $J(x,f)$ for the Jacobian determinant. A mapping $f \in W^{1,n}_{\loc}(\R^n ; \mathbb R^n)$ (not necessarily homeomorphism) satisfying~\eqref{Kav22} is called $K$-quasiregular~\cite{Rebook, Ribook}.
\end{definition}

 Our study of the variation of quasiconformal mappings on lines grew out of~\cite{IKO3} where it was proved that the ordinary differential equation $\dot{x}=f(x)$ has unique local solutions outside of $f^{-1}(0)$ provided that $f$ is quasiconformal and has bounded variation on $C^1$-smooth curves.
 By Gehring's theorem~\cite{Ge73} quasiconformal mappings are locally in $W^{1,s}$ for some $s>n$ and therefore
have finite $p$-variation on lines for some $p<n$. (In fact, any homeomorphism of class $W^{1,n}$ has finite $n$-variation on
lines~\cite[Thm 4.3]{Ma99}). In the opposite direction, a theorem of Bishop~\cite[Thm 1.1]{Bi}
implies that for any $p<n$ there is a quasiconformal mapping $f\colon\R^n\to\R^n$ such that the image of some line
segment under $f$ has Hausdorff dimension greater than $p$. Clearly, such $f$ has infinite $p$-variation on this segment.

Interestingly, some classes of quasiconformal mappings exhibit much higher regularity
along lines and smooth curves than their Sobolev or H\"older regularity would suggest.

\begin{definition}\label{demondef}
A mapping $f \colon \R^n \to \R^n$ is called \textit{$\delta$-monotone}, $0<\delta \le 1$, if for every $a,b \in \R^n$
\begin{equation}\label{demon}
\inn{f(a)-f(b),a-b} \le \delta \abs{f(a)-f(b)}\abs{a-b}.
\end{equation}
\end{definition}

Any nonconstant $\delta$-monotone mapping is quasiconformal~\cite[Cor. 7]{Ko}.  For example, the radial stretch $f(x)=|x|^{\alpha-1} x$, where $\alpha >0$, is $\delta$-monotone for some $\delta= \delta(\alpha)$.
This mapping is locally  H\"older continuous with exponent $\min\{\alpha,1\}$, which can be arbitrarily close to $0$. This shows that $\delta$-monotone mappings are no more regular on the H\"older and Sobolev scales  than general quasiconformal mappings. However, they have bounded variation on $C^1$-smooth curves, see~\cite[Thm. 3.11.7]{AIMbook} and~\cite[Thm. 1.10]{IKO3}. In particular,
\begin{equation}\label{property}
\var_{[a,b]}(f;1)<\infty \quad \mbox{if } a,b \in \R^n.
\end{equation}

When $n=2$, we often identify $\R^n$ with $\C$ and use the complex derivatives $f_z$ and $f_{\bar z}$.
Then the inequality~\eqref{Kav22} reads as
\begin{equation}\label{Kav23}
\left|f_{\bar z} \right| \leqslant k \left|f_{ z} \right|\,  \quad \textnormal{a.e.,} \quad  \textnormal{where } \  k= \frac{K-1}{K+1}
\end{equation}

A $\delta$-monotone mapping $f \colon \C \to \C$ satisfies the stronger, \textit{reduced} distortion inequality
\begin{equation}\label{BeIn}
\abs{f_{\bar z}} \le k \re f_z \quad \textnormal{a.e. in } \C,
\end{equation}
for some constant $0<k<1$ (Theorem~3.11.6~\cite{AIMbook}). The converse is false: for instance, $f(z)=iz$ satisfies~\eqref{BeIn} but is not $\delta$-monotone.

\begin{definition}\label{redqcdef}
A homeomorphism $f\in W_{\loc}^{1,2}(\C;\C)$ is called \textit{reduced quasiconformal} if it satisfies~\eqref{BeIn}.
\end{definition}

Inequality~\eqref{BeIn} implies that $f$ is a solution of the \textit{reduced Beltrami equation}
\begin{equation}\label{BeEq}
 f_{\bar z} = \lambda(z)  \re f_z \quad \textnormal{a.e. in } \C
\end{equation}
where $\abs{\lambda (z)}\le k$. Conversely, it is shown in~\cite[Thm 6.3.2]{AIMbook}
that any homeomorphic  solution of~\eqref{BeEq}  has constant sign of $\re f_z$. Therefore such solutions satisfy~\eqref{BeIn} up to a change of sign.

Unlike quasiconformality, the properties~\eqref{demon} and~\eqref{BeIn} are preserved under addition.
Both of these classes arise naturally in the theory of elliptic partial differential equations~\cite{AN,AIMbook,AJ}.
Our first result shows that~\eqref{property} cannot be extended to reduced quasiconformal mappings.

\begin{theorem}\label{rBexample} For every $k\in (0,1)$ there exists a reduced quasiconformal mapping
$f \colon \mathbb \C \to \mathbb C$ that satisfies~\eqref{BeIn} but does not have bounded variation on any nontrivial interval $[a,b]\subset\R$.
Furthermore, $f$ can be chosen so that $f_{\bar z}, \re f_z\in L^{\infty}(\C)$ and $\re f(x)=x$ for all $x\in \R$.
\end{theorem}

In other words, $f$ in Theorem~\ref{rBexample} maps $\R$ into a curve that is nowhere locally rectifiable.
Although such examples of (non-reduced) quasiconformal mappings were known for a long time~\cite{Tu81},
our mapping $f$ seems to be the first one given by an explicit analytic expression, see~\eqref{rBex1}. The additive property of reduced quasiconformal
mappings allows us to derive the following result from Theorem~\ref{rBexample}.

\begin{corollary}\label{countlines} For any countable family of parallel lines $\{L_j  \colon j=1,2,\dots  \} $ in $\C$ there exists a quasiconformal
mapping $f\colon\C\to\C$ such that any nontrivial subarc of $f(L_j)$ is unrectifiable  for every $j$.
\end{corollary}

Corollary~\ref{countlines} exhibits a quasiconformal mapping with irregular behavior on a relatively large set. The authors of~\cite{BHS}
asked (Question 4.4) whether for any set $E\subset\C$ of planar measure zero there is a quasiconformal mapping $f\colon\C\to\C$ such that
the volume derivative $\lim\limits_{r\to 0}\frac{\abs{f(B(x,r)}}{\abs{B(x,r)}}$ is infinite at every point of $E$.
While the singular behaviour of $f$ in Corollary~\ref{countlines} is of different nature,
the additive property of reduced quasiconformal mappings can be potentially useful in creating mappings with a large set of
infinite volume derivative.

Our second main result shows that, Theorem~\ref{rBexample} nonwithstanding, reduced quasiconformal mappings are much more regular
on lines than general quasiconformal mappings. In particular, they have finite $p$-variation for any $p>1$.

\begin{theorem}\label{fvar} Let $f\colon \C\to \C$ is a reduced quasiconformal mapping.
Then for any $q>1$ the mapping $f$ has finite $\phi$-variation on lines with
\begin{equation}\label{fvar1}
\phi(t)=\frac{t}{(\log (e+1/t))^q} \ \mbox{ if } t>0 \quad \mbox{ and }\quad  \phi(0)=0.
\end{equation}
\end{theorem}

The conclusion of Theorem~\ref{fvar} is false for $0\le q<1/2$, see Remark~\ref{negphi}.
The gap between exponents $1/2$ and $1$ remains open.

\begin{question}\label{questionq} What is the smallest value of $q$ for which the conclusion of Theorem~\ref{fvar} holds?
\end{question}

Since $\delta$-monotone mappings exist in any dimension $n\ge 2$, one may ask whether it is possible to extend the
definition of reduced quasiconformal mappings to higher dimensions. This question is addressed in section~\ref{higher}, where we use quaternions
to define reduced quasiconformal mappings in four dimensions, and extend Theorem~\ref{fvar} to them.

\begin{question}\label{questionred} Is there a natural analogue of reduced
quasiconformal mappings in dimensions other than $2$ and $4$?
\end{question}

\section{Preliminaries}

In this section we first estimate the $p$-variation of Sobolev functions of lines. Although this result is probably known we  give a proof for the sake of completeness.  Later in the section we define quasisymmetric and monotone mappings and introduce some relevant notation.
In this paper $\Omega$ stands for a domain in $\R^n$.

\begin{proposition}\label{morrey}
Let  $u\in W^{1,s}(\Omega)$,  $s>n$. Then the restriction of $u$ to any closed line segment  $I \subset \Omega$ has finite $p$-variation with $p=s/(s-n+1)$.
\end{proposition}

The Morrey-Sobolev embedding theorem states that $ W^{1,s}(\Omega) \subset C^\alpha_{\loc} (\Omega)$ with $\alpha= 1- n/s$. Clearly, any function $u \in C^\alpha_{\loc} (\Omega)$ has finite $p$-variation on lines with $p=1/\alpha= s/(s-n)$. However, Proposition~\ref{morrey} gives a better value of $p$. Its proof requires the following lemma.

\begin{lemma}\label{union}
Let $I$ be a line segment partitioned into smaller segments $I_m$, $m=1,2, \dots , M$. For any  mapping $f \colon I \to \R^n$ we have
\begin{equation}
\var_{I} (f;  \phi) \le  \sum_{m=1}^M \var_{I_m} (f;\varphi)+ (M-1) \osc_I f .
\end{equation}
\end{lemma}
\begin{proof}
Fix a partition $(a_j)_{j=0}^{N}$ of $I$. We divide the set of indices as follows.
\[E= \{j=1, \dots , N \colon [a_{j-1},a_j] \subset I_m \text{ for some } m   \} , \qquad F= \{1, \dots , N\} \setminus E .\]
Then
\[\sum_{j\in E} \phi (\abs{f(a_j)-f(a_{j-1})}) \le \sum_{m=1}^M \var_{I_m} (f;\varphi)\]
and
\[\sum_{j\in F} \phi (\abs{f(a_j)-f(a_{j-1})}) \le (M-1) \osc_I f . \]
\end{proof}

\begin{proof}[Proof of Proposition~\ref{morrey}.]
Dividing $I$ into subintervals and using Lemma~\ref{union} we may reduce our task to the case
\[\diam I < \dist (I, \partial \Omega) .\]
Let $(a_j)_{j=0}^N$ be a partition of  $I$.  For $j=1, \dots , N$ let $\overline{B}_j$ be the closed ball with segment $[a_{j-1}, a_j]$ as a diameter.  Morrey's inequality~\cite[p.~143]{EG} yields
\[\osc_{\overline{B}_j} u \le C\,  \abs{a_j - a_{j-1}}^{1-n/s} \left(\int_{B_j} \abs{\nabla u(x)}^s \, d x \right)^{1/s} . \]
Raising to the power $p$ and noticing that $(1-n/s)p= 1-p/s$ we arrive at
\[\big( \osc_{\overline{B}_j} u \big)^p \le C\,  \abs{a_j - a_{j-1}}^{1-p/s} \left(\int_{B_j} \abs{\nabla u(x)}^s \, d x \right)^{p/s} . \]
Summing over $j$ and  applying  H\"older's inequality we obtain
 \begin{equation}\label{E}
\sum_{j=1}^N  \big(\osc_{\overline{B}_j} u\big)^p  \le C \left( \sum_{j=1}^N  \abs{a_j - a_{j-1}}   \right)^{1-p/s}   \left( \sum_{j=1}^N \int_{B_j} \abs{\nabla u(x)}^s \, d x \right)^{p/s}.
 \end{equation}
Therefore
\[
 \var_I(u; p) \le C (\diam I)^{1-p/s}  \left( \int_{\Omega} \abs{\nabla u(x)}^s \, d x  \right)^{p/s}
 \]
as desired.
\end{proof}

\begin{definition}\label{qs}
Let $\eta \colon [0, \infty) \to [0, \infty)$ be a homeomorphism. An injective mapping $f \colon \R^n \to \R^n$ is $\eta$-quasisymmetric if
\[\frac{\abs{f(c)-f(a)}}{\abs{f(b)-f(a)}} \le \eta \left(\frac{\abs{c-a}}{\abs{b-a}}  \right)\]
for any distinct points $a,b,c \in \R^n$.
The function $\eta$ is called a modulus of quasisymmetry of $f$.
\end{definition}
It is well-known that a mapping $f \colon \R^n \to \R^n$ is quasiconformal if and only if it is sense-preserving and quasisymmetric ~\cite{Hebook}.

Given a mapping $f \colon \R^n \to \R^n$, where $\R^n \subset\R^n$, we define
the modulus of monotonicity $\Delta_f \colon \R^n \times \R^n \to \R$ by the rule
\begin{equation}\label{Deltadef}
\Delta_f(a,b)= \begin{cases} \left\langle f(a)-f(b), \frac{a-b}{|a-b|} \right\rangle & \quad \textnormal{if } a \ne b\\
0 & \quad \textnormal{if } a = b  \end{cases}
\end{equation}
Clearly $\abs{\Delta_f(a,b)}\le\abs{f(a)-f(b)}$. By definition, $f$ is a monotone mapping if $\Delta_f(a,b)\ge 0$
for all $a,b\in\R^n$, and is strictly monotone if $\Delta_f(a,b)>0$ unless $a=b$.
Any reduced quasiconformal is monotone by (1.9) in~\cite{IKO3}.
Also, $f$ is $\delta$-monotone
if and only if  $\Delta_f(a,b)\ge\delta\abs{f(a)-f(b)}$ for all $a,b \in \R^n$.
When $n=2$, the modulus of monotonicity can be expressed in complex notation:
\[\Delta_f(a,b)=\re\left(\frac{f(a)-f(b)}{a-b}\right)\abs{a-b}.\]

\section{Generalized variation on lines: Proof of Theorem~\ref{fvar}}

We will obtain Theorem~\ref{fvar} as a consequence of the following result.

\begin{theorem}\label{fvgen} Let $f\colon \Omega\to \R^n$ be a mapping and suppose that there
is a homeomorphism $\eta\colon[0,\infty) \to[0,\infty)$ such that
\begin{equation}\label{fvgen1}
\abs{f(c)-f(a)}\le \frac{\abs{c-a}}{\abs{b-a}}\abs{f(b)-f(a)} +
\eta\left(\frac{\abs{c-a}}{\abs{b-a}}\right)\Delta_f(a,b)
\end{equation}
for any distinct points $a,b,c\in\Omega$. Then for any $q>1$ the mapping $f$ has finite $\phi$-variation
on lines with $\phi$ as in~\eqref{fvar1}.
\end{theorem}

Before proving Theorem~\ref{fvgen} we derive Theorem~\ref{fvar} from it.

\begin{proof}[Proof of Theorem~\ref{fvar}] Let $f\colon\C \to\C$ be a reduced quasiconformal
mapping. We may assume that $f$ is nonlinear. For any $\lambda \in \R $ the mapping $f^{\lambda}(z)=f(z)+i\lambda z$ also satisfies the reduced distortion inequality~\ref{BeIn} with the same constant $k$ as $f$. By~\cite[Cor. 1.5]{IKO} $f^\lambda$ is a homeomorphism. Therefore, $f^\lambda$ is   $K$-quasiconformal with $K$ independent of $\lambda$.  Since quasiconformality implies quasisymmetry in $\C$~\cite[Thm 11.14]{Hebook}, there is a homeomorphism  $\eta\colon[0,\infty) \to[0,\infty)$ such that  $f^\lambda$ is $\eta$-quasisymmetric in $\C$ for all $\lambda \in \R$.   Given distinct points $a,b,c\in \R$, let $\lambda=-\im\frac{f(b)-f(a)}{b-a}$, so that $\abs{f^{\lambda}(b)-f^{\lambda}(a)}=\Delta_f(a,b)$.  Since   $f^{\lambda} $   is $\eta$-quasisymmetric,  we have
\[
\abs{f^{\lambda}(c)-f^{\lambda}(a)}\le \eta\left(\frac{\abs{c-a}}{\abs{b-a}}\right)\abs{f^{\lambda}(b)-f^{\lambda}(a)},
\]
from which~\eqref{fvgen1} follows by means of the triangle inequality. It remains to apply Theorem~\ref{fvgen} to $f$.
\end{proof}

Our proof of Theorem~\ref{fvgen} is based on two lemmas.

\begin{lemma}\label{vargrowth} Let $f\colon \Omega\to\R^n$ be as in Theorem~\ref{fvgen}. Given
distinct points $a,b\in \Omega$ and a partition $0=t_0<t_1<\dots<t_N=1$, let $a_j=a+t_j(b-a)$, $j=0,\dots,N$.
If $a_j\in\Omega$ for all $j$, then
\begin{equation}\label{vargrowth1}
\sum_{j=1}^N \abs{f(a_j)-f(a_{j-1})}\le C\log (N+1)\abs{f(a)-f(b)}
\end{equation}
where the constant $C$ depends only on $\eta$ in~\eqref{fvgen1}.
\end{lemma}

\begin{proof} It suffices to prove~\eqref{vargrowth1} for $N=2^m$. For
$j=1,\dots,2^{m-1}$ we apply~\eqref{fvgen1} to the points $a_{2j}, a_{2j-1}, a_{2j-2}$ and find that
\[
\abs{f(a_{2j})-f(a_{2j-1})}\le \frac{\abs{a_{2j}-a_{2j-1}}}{\abs{a_{2j}-a_{2j-2}}}\abs{f(a_{2j})-f(a_{2j-2})}
+\eta(1)\Delta_f(a_{2j},a_{2j-2})
\]
and
\[
\abs{f(a_{2j-1})-f(a_{2j-2})}\le \frac{\abs{a_{2j-1}-a_{2j-2}}}{\abs{a_{2j}-a_{2j-2}}}\abs{f(a_{2j})-f(a_{2j-2})}
+\eta(1)\Delta_f(a_{2j},a_{2j-2}).
\]
Adding these inequalities we obtain
\[
\begin{split}
 \abs{f(a_{2j})-f(a_{2j-1})} & + \abs{f(a_{2j-1})-f(a_{2j-2})} \\ & \le
\abs{f(a_{2j})-f(a_{2j-2})}+2\eta(1)\Delta_f(a_{2j},a_{2j-2}).
\end{split}
\]
Observe that $\Delta_f(x,y)+\Delta_f(y,z)= \Delta_f(x,z)$ for all $y\in [x,z]$.
Therefore, summation over $j=1,\dots,2^{m-1}$ yields
\begin{equation}\label{vargrowth9}
\sum_{j=1}^{2^m} \abs{f(a_j)-f(a_{j-1})}\le \sum_{j=1}^{2^{m-1}} \abs{f(a_{2j})-f(a_{2j-2})}+2\eta(1)\Delta_f(a,b).
\end{equation}
Notice that the sum on the right hand side involves only even indices.
Next we apply~\eqref{vargrowth9} to the partition $a_0,a_2,\dots,a_{2^m}$ to rarefy it further.
After $m$ steps we arrive at
\[
\sum_{j=1}^{2^m} \abs{f(a_j)-f(a_{j-1})}\le \abs{f(a)-f(b)}+2m\eta(1)\Delta_f(a,b)\le (2m+1)\abs{f(a)-f(b)}.
\]
This completes the proof of~\eqref{vargrowth1}.
\end{proof}

The following lemma will allow us to derive the conclusion of
Theorem~\ref{fvgen} from the growth estimate~\eqref{vargrowth1}.

\begin{lemma}\label{series}
Suppose that $d_1\ge d_2\ge \dots\ge d_N$ are positive numbers such that for $j=1,\dots,N$
the partial sum $s_j:=d_1+\dots+d_j$ is bounded by $C\log (j+1)$, where $C$ is a constant.
Let $q>1$ and the function $\phi$ be defined by the formula~\eqref{fvar1}. Then
\begin{equation}\label{series1}
\sum_{j=1}^N \phi(d_j)\le C'
\end{equation}
where $C'$ depends only on $C$ and $q$.
\end{lemma}

\begin{proof} Since $d_j\le C\log(j+1)/j\le C\sqrt{j}$, it follows
that $\log(e+1/d_j)\ge C_1\log(j+1)$, where $C_1$ depends only on $C$. Therefore,
\begin{equation}\label{series2}
\sum_{j=1}^N \phi(d_j)\le C_1^{-q}\sum_{j=1}^N \frac{d_j}{(\log(j+1))^{q}}.
\end{equation}
Next we use summation by parts, replacing $d_j$ with $s_j-s_{j-1}$, where $s_0=0$ by convention.
\begin{equation}\label{series3}
\begin{split}
&\sum_{j=1}^N \frac{d_j}{(\log(j+1))^{q}} \\
&=\frac{s_N}{(\log(N+1))^{q}}+
\sum_{j=1}^{N-1}s_j\left(\frac{1}{(\log(j+1))^{q}}-\frac{1}{(\log(j+2))^{q}}\right)
\end{split}
\end{equation}
The first term on the right is bounded by $C/(\log (N+1))^{q-1}$.
Since $s_j\le C\log(j+1)$ and
\[\frac{1}{(\log(j+1))^{q}}-\frac{1}{(\log(j+2))^{q}}\le
\frac{q}{(j+1)(\log(j+1))^{1+q}},\]
it follows that
\begin{equation}\label{series4}
\begin{split}
&\sum_{j=1}^{N-1}s_j\left(\frac{1}{(\log(j+1))^{q}}-\frac{1}{(\log(j+2))^{q}}\right) \\
&\le C\sum_{j=1}^{\infty}\frac{1}{(j+1)(\log(j+1))^{q}}=:C_2
\end{split}
\end{equation}
where $C_2$ depends only on $C$ and $q$. Combining~\eqref{series2}, \eqref{series3},
and \eqref{series4}, we obtain~\eqref{series1}.
\end{proof}

\begin{proof}[Proof of Theorem~\ref{fvgen}]
Let $a,b\in\R^n$ be two distinct points.
Given a partition $(a_j)_{j=0}^N$ of $[a,b]$, let $d_1 \ge \dots \ge d_N$ be the numbers  $\abs{f(a_j)-f(a_{j-1})}$ arranged in the nonincreasing order.  By Lemma~\ref{vargrowth}  the partial sums $s_j=d_1+ \dots +d_j$ are bounded by   $C\log (2j+2)\abs{f(a)-f(b)}$, where $C$ is the constant in~\eqref{vargrowth1}.
Applying Lemma~\ref{series} we arrive at the conclusion of the theorem.
\end{proof}

\section{Failure of bounded variation: Proof of Theorem~\ref{rBexample}}

\begin{proof}
Let $Q\subset \C$ be the open square $\{x+iy\colon 2< x< 6, \abs{y}< 2\}$.
The closure of $Q$ contains two smaller closed squares $Q_1=\{x+iy\colon \abs{x-3}+\abs{y}\le 1\}$ and
$Q_2=\{x+iy\colon \abs{x-5}+\abs{y}\le 1\}$.

\begin{figure}[!h]
\begin{center}
\psfrag{q}{\small {${Q}$}}
\psfrag{q1}{\small ${Q_1}$}
\psfrag{q2}{\small ${Q_2}$}
\psfrag{2}{\small ${2}$}
\psfrag{4}{\small ${4}$}
\psfrag{6}{\small ${6}$}

\includegraphics*[height=2.0in]{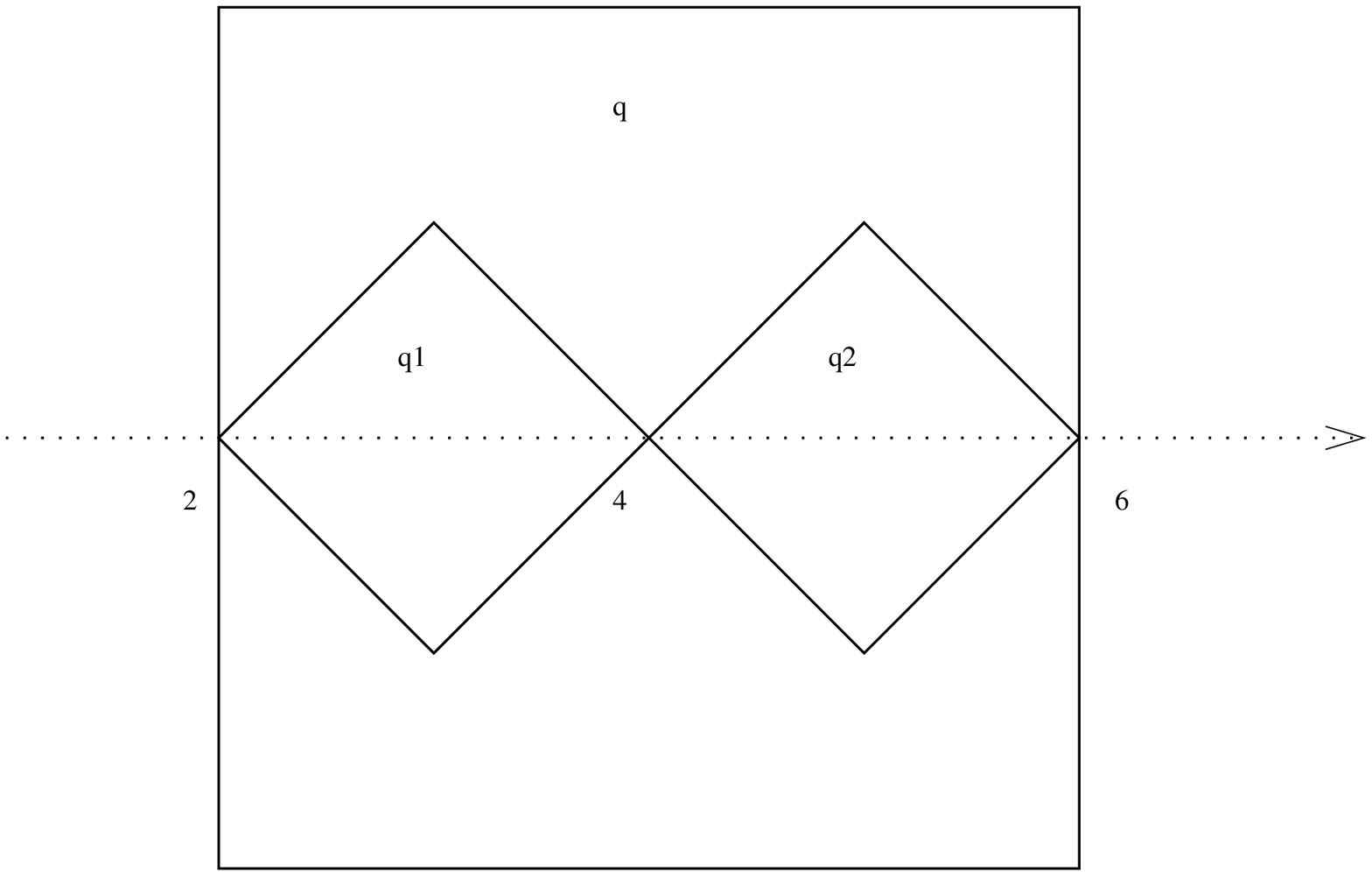}
\caption{}
\end{center}
\end{figure}

Let $g\colon \overline{Q}\to \C$ be a Lipschitz function such that
\[
g(z)=\begin{cases} i(z-2),\quad &z\in Q_1; \\
i(6-z), \quad &z\in Q_2; \\
0,\quad &z\in \partial Q.
\end{cases}
\]
Extend $g$ to the set $A=\bigcup_{k\in\Z}(\overline{Q}+8k)$ so that $g(z+8)=g(z)$. Finally, set $g(z)=0$ for $z\notin A$.

Let $L$ be the Lipschitz constant of $g$. We shall prove that for $0<\epsilon<1/(2L)$ the mapping
\begin{equation}\label{rBex1}
f(z)=z+\epsilon\sum_{m=0}^{\infty}4^{-m}g(4^m z),\quad z\in\C,
\end{equation}
satisfies
\begin{equation}\label{rBexdiff}
 \left| f_{\bar z} \right| \leqslant k \re f_z \quad \textnormal{a.e. } \quad \textnormal{ with } \quad  k=\frac{\epsilon L}{1-\epsilon L}.
 \end{equation}
First of all, the series in~\eqref{rBex1} converges uniformly because $g$ is bounded.
Let
\[B=\bigcup_{\ell \in\Z} \big\{  [Q\setminus(Q_1\cup Q_2)]+8\ell \big\} \]
and note that
\begin{equation}\label{Bprop2}
B\subset \bigcup_{j\in \Z}\{z\colon \abs{\re z-2j}<\abs{\im z}\}
\end{equation}
and
\begin{equation}\label{Bprop3}
B\cap \bigcup_{j\in \Z}\{z\colon \abs{\re z-8j}<\abs{\im z}\}=\varnothing.
\end{equation}
We claim that that for any $z\in \C$ there exists at most one integer $m\ge 0$ such that $4^m z\in B$.
Indeed, let $m_0$ be the smallest such integer. Replacing $z$ with $4^{m_0}z$, we may assume that $m_0=0$, i.e., $z\in B$.
According to~\eqref{Bprop2}, there exists $j\in\Z$ such that $\abs{\re z-2j}<\abs{\im z}$.
For any $m\ge 1$ the number $\zeta=4^mz$ satisfies
$\abs{\re \zeta-8\cdot 4^{m-1}j}<\abs{\im  \zeta}\}$, which implies $\zeta\notin B$ by virtue of~\eqref{Bprop3}.
This proves the claim.

Since both $\re g_z$ and $g_{\bar z}$ vanish a.e. outside of $B$, it follows that
\[
{ 1-\epsilon L \le \re f_z (z)} \le 1+ \epsilon L    ,\quad \abs{f_{\bar z}(z)} \le \epsilon L \quad \mbox{a.e. in } \C.
\]
This proves~\eqref{rBexdiff}. Since $g(z)=0$ when $\abs{\im z}>2$, it follows that
$|\im f_z(z)|\le \epsilon \, m  L$ when $\abs{\im z}\ge 2\cdot 4^{-m}$. This implies
\begin{equation}\label{Df}
\abs{Df(z)} \le C \log \left(e+ 1/\abs{\im z }\right)
\end{equation}
for some constant $C$. Hence $f\in W^{1,p}_{\textrm{loc}}(\C;\C)$
for all $p<\infty$. Therefore, $f$ is quasiregular. By~\cite[Cor. 1.5]{IKO} it is quasiconformal.
It is clear that $\re f(x)=x$ for all $x\in \R$.

It remains to prove that the function
\[h(x):=\epsilon^{-1}\im f(x)=-i \sum_{m=0}^{\infty}\frac{1}{4^{m}}g(4^mx),\quad x\in\R\]
has infinite variation on any nontrivial interval $[a,b]\subset\R$.
Due to the self-similar structure of $h$ it suffices to consider the interval $[0,8]$.
We will show that the sum
\[
V_N:=\sum_{j=0}^{4^N} \left|h\left(\frac{8j}{4^{N}}\right)-h\left(\frac{8j-8}{4^{N}}\right)\right|
\]
satisfies
\begin{equation}\label{nonr1}
V_N\ge c\sqrt{N}
\end{equation}
with an absolute constant $c>0$. Let $x_j=(8j)4^{-N}$, $j=0,\dots, 4^N$.
When $m\ge N$, we have $g(4^m x_j)=0$ for all $j$.
Therefore, in the definition of $V_N$ we can replace $h$ with the partial sum
\[h_N(x)=-i \sum_{m=0}^{N-1}\frac{1}{4^{m}}g(4^mx).\]
Since $h_N$ is affine on each interval $[x_{j-1},x_j]$, it follows that
$V_N=\displaystyle \int_0^8 \abs{h_N'(x)}\,dx$. We claim that for a.e. $x\in\R$
\begin{equation}\label{derg}
\frac{d}{dx} \left(-i g(x)\right) = \frac{1}{2} (s_{0}(x/8)-s_{1}(x/8)),
\end{equation}
where $s_m$ is the $m$th Rademacher function, defined by \[s_m(x)=\sign \sin(2^{m+1}\pi x).\]
Since both sides of~\eqref{derg} are periodic functions with period $8$, it suffices to check that equality holds a.e. on the interval $(0,8)$.
From the definitions of $g$ and $s_m$ one can see that that both sides of~\eqref{derg}
agree with $\chi_{[2,4]}-\chi_{[4,6]}$ when $0<x<8$ and $x\ne 2,4,6$. This proves~\eqref{derg}.
It then follows that
\[h_N'(x)=\frac12\sum_{m=0}^{N-1}(s_{2m}(x/8)-s_{2m+1}(x/8)).  \]
The $L^1$ norm of a Rademacher series  is comparable to the $\ell^2$ norm of its coefficients~\cite[Thm.  V.8.4]{Zy}.
Hence  $\displaystyle \int_0^8 \abs{h_N'(x)}\,dx\ge c\sqrt{N}$ with an absolute constant $c>0$.
This completes the proof of~\eqref{nonr1}.
\end{proof}

\begin{remark}\label{negphi} The mapping $f$ constructed in the proof of Theorem~\ref{rBexample}
does not have finite $\phi$-variation on lines with $\phi$ as in~\eqref{fvar1} for $0\le q<1/2$.
\end{remark}
\begin{proof} Using Jensen's inequality and the estimate~\eqref{nonr1}, we obtain
\[
\begin{split}
\sum_{j=0}^{4^N} \phi(\abs{h(x_j)-h(x_{j-1})})&\ge 4^N\phi(V_N/4^N) \\
&\ge \frac{c\sqrt{N}}{(\log(e+4^N/(c\sqrt{N}))^{q}}\to\infty
\end{split}
\]
as $N\to\infty$.
\end{proof}

\begin{proof}[Proof of Corollary~\ref{countlines}]  We may assume that the lines $L_j$ are parallel to the the real axis; that is, $L_j= \{z \in \C \colon \im z=b_j\}$ where $b_j$ are distinct real numbers. For $m=1,2,\dots$ let $\epsilon_m= \min_{1 \le j < \ell \le m} \abs{b_j- b_\ell}$ and choose $c_m > 0$ so that
\begin{equation}\label{cm1}
c_m \abs{b_m} < 2^{-m}
\end{equation}
and
\begin{equation}\label{cm2}
c_m \log \left(e + 1/\epsilon_m\right) < 2^{-m}.
\end{equation}
We define
\begin{equation}\label{mapF}
F(z)= \sum_{m=1}^\infty c_m f(z-ib_m)
\end{equation}
where $f$ is the  mapping in~\eqref{rBex1}. Note that $\abs{f(z)} \le \abs{z}+M$ for some constant $M$. The sum in~\eqref{mapF} converges locally  uniformly  because by~\eqref{cm1}--\eqref{cm2}
\[
\abs{ c_m f(z-ib_m)} \le c_m \left(\abs{z}+ \abs{b_m}+M\right)\le 2^{-m} \left( \abs{z} +1 +M\right) .
\]
Therefore $F$ is quasiconformal~\cite[II 5.3]{LV}. We claim for every $j=1,2,\dots$ the sum
\[R_j(z)=\sum_{m\ne j}  c_m f(z-ib_m)\]
is  Lipschitz on the line $L_j$. By~\eqref{Df} the restriction of $f(z-ib_m)$ to $L_j$ is Lipschitz with a constant $C\log \left( e+ 1/\abs{b_m-b_j} \right)$. We estimate the Lipschitz constant of $R_j$  by
\[ 
\sum_{m \ne j} c_m \log \left( e+ 1/\abs{b_m-b_j} \right)  \le \sum_{m<j} c_m \log \left( e+ 1/ \epsilon_j\right) +  \sum_{m>j} c_m \log \left( e+ 1/ \epsilon_m \right)  .
\]
The first sum on the right has finitely many terms, and the second sum converges by~\eqref{cm2}.  
 Since $F(z)=c_j f(z-ib_j)+R_j$, it follows that $F(L_j)$ is not locally rectifiable at any of its points.
\end{proof}

\section{Reduced quasiconformal mappings in four dimensions}\label{higher}

Our first goal in this section is to reformulate the definition of reduced quasiconformal mappings (Definition~\ref{redqcdef})
in terms of differential matrices $Df(x)\in \M_n$ rather than complex derivatives. This is done in Proposition~\ref{realred} below. We write $\M_n$ for the set of all $n\times n$ matrices with real entries. Also, for $\delta\in [-1,1]$ we
let
\[\M_n(\delta)=\{A\in\M_n\colon \inn{Av,v}\ge \delta \abs{Av}\abs{v} \text{ for all }
v\in\R^n\}.\]

\begin{proposition}\label{deltaqc} For $\delta\in (0,1)$ let $H(\delta)= \frac{1+ \sqrt{1-\delta^2}}{1- \sqrt{1-\delta^2}}$. Then
$\norm{A}\norm{A^{-1}}\le H(\delta)$ for all nonzero matrices $A\in \M_n(\delta)$.
\end{proposition}
\begin{proof}
For $n=2$ this proposition was proved in~\cite[p.84]{AIMbook}. If $n \ge 3$, let $v$ and $w$ be distinct unit  vectors such that $\abs{Av}= \norm{A}$ and $\abs{Aw}= \norm{A^{-1}}^{-1}$. Applying the two-dimensional case to the subspace spanned by $v$ and $w$, we arrive at the desired conclusion.
\end{proof}

A matrix $A\in \M_2$ determines a linear mapping $x\mapsto Ax$ of the plane $\R^2$. The same linear mapping
can be written as $z\mapsto \alpha^+ z+\alpha^-\bar z$ for some $\alpha^+,\alpha^- \in\C$. The numbers $\alpha^+$ and $\alpha^-$ can be thought of as the conformal and anticonformal parts of $A$ (cf.~\cite{FS}).

\begin{proposition}\label{compdelta}~\cite[Thm. 3.11.6]{AIMbook} A matrix $A$ belongs to
$\M_2(\delta)$ if and only if $\abs{\alpha^-}+\delta\abs{\im \alpha^+}\le \sqrt{1-\delta^2}\re \alpha^+$.
\end{proposition}

A complex number $a+ib\in\C$ can be identified with the $2\times 2$ matrix
$Z=\left(\begin{smallmatrix}a & -b \\ b & a\end{smallmatrix}\right)$.
Thus we may consider $\C$ as a linear subspace of $\mathbb{M}_2$.
Within this subspace each matrix decomposes into real and imaginary parts:
$\re Z = \left(\begin{smallmatrix}a & 0 \\ 0 & a\end{smallmatrix}\right)$ and
$\im Z = \left(\begin{smallmatrix}0 & -b \\ b & 0\end{smallmatrix}\right)$.
Given $A\in\mathbb M_2$, let $\C(A)$ be the orthogonal projection of $A$ onto the subspace $\C$.

\begin{proposition}\label{realred} A mapping $f\in W_{\loc}^{1,2}(\R^2;\R^2)$ is reduced
quasiconformal in the sense of Definition~\ref{redqcdef} if and only if there exists $\delta>0$ such that for a.e. $x\in\R^2$ the derivative $A=Df(x)$ satisfies $A-\im \C(A)\in \mathbb M_2(\delta)$.
\end{proposition}
\begin{proof} Writing the matrix $A=Df(x)$ in conformal-anticonformal coordinates as $(\alpha^+,\alpha^-)$, we observe
that $A-\im \C(A)$ corresponds to $(\re \alpha^+,\alpha^-)$. According to Proposition~\ref{compdelta}, the condition
$A-\im \C(A)\in \mathbb M_2(\delta)$ is equivalent to $\abs{\alpha^-}\le \sqrt{1-\delta^2}\re \alpha^+$.
The latter is inequality is the same as~\eqref{BeIn} with $k=\sqrt{1-\delta^2}$.
\end{proof}

A quaternion $\alpha+\beta\mathbf{i}+\gamma\mathbf{j}+\zeta\mathbf{k}$ can be identified with a $4\times 4$ real matrix
\begin{equation}\label{quat1}
Q=\begin{pmatrix} \alpha&-\beta&-\gamma&-\zeta \\ \beta&\alpha&-\zeta&\gamma \\ \gamma&\zeta&\alpha&-\beta \\  \zeta&-\gamma&\beta&\alpha
\end{pmatrix}
\end{equation}
With this identification we consider the set of quaternions $\mathbb H$ as a subset of $\M_4$.
Since quaternion conjugation corresponds to matrix transposition, we have $Q^TQ=\norm{Q}^2I$, where
$\norm{Q}$ is the operator norm of matrix $Q$, also equal to the absolute value of the quaternion.
Consequently, $\abs{Qv}=\norm{Q}\abs{v}$ for any vector $v\in\R^4$.

A quaternion $Q$ is the sum of its real (scalar) and imaginary parts:
\begin{equation}\label{quat1n}
\re Q=\begin{pmatrix} \alpha&0&0&0 \\ 0&\alpha&0&0 \\ 0&0&\alpha&0 \\  0&0&0&\alpha
\end{pmatrix},\quad \im Q=Q-\re Q.
\end{equation}
If $\re Q=0$, the quaternion $Q$ is purely imaginary.
For a matrix $A\in\mathbb M_4$, we define $\HH(A)$ to be the orthogonal projection of $A$ onto the subspace
$\HH\subset \mathbb M_4$.

\begin{definition}\label{quatred} A homeomorphic mapping $f\in W_{\loc}^{1,4}(\R^4;\R^4)$
is \textit{reduced quasiconformal} if there exists $\delta>0$ such that for a.e. $x\in\R^4$ the derivative
$A=Df(x)$ satisfies $A-\im \HH(A)\in \mathbb M_4(\delta)$.
\end{definition}

First of all, we need to justify the terminology by proving the following proposition.

\begin{proposition}\label{quatprop} Any reduced quasiconformal mapping $f\colon\R^4\to\R^4$ is
$K$-quasiconformal, where $K$ depends only on $\delta$ in Definition~\ref{quatred}. In addition, $f$ is monotone.
\end{proposition}

\begin{proof} The essence of this proposition is the algebraic implication
\begin{equation}\label{quatp1}
A-\im \HH(A)\in \mathbb M_4(\delta) \implies  \norm{A} \norm{A^{-1}} \le \widetilde{H}(\delta).
\end{equation}
 We may assume that $A$ is a nonzero matrix.
Let $Q=\im \HH(A)$ and $B=A-Q$.   If $Q=0$, then Proposition~\ref{deltaqc} gives~\eqref{quatp1} with $ \widetilde{H}(\delta)= H(\delta)$. Assume $Q\ne 0$. Fix a unit vector $v\in\R^4$.
Since $Qv/\norm{Q}$ is a unit vector orthogonal to $v$, it follows that
\[\inn{Bv,v}^2+\norm{Q}^{-2}\inn{Bv,Qv}^2\le \abs{Bv}^2\]
Using the inequality $\inn{Bv,v}\ge \delta \abs{Bv}$, we obtain
\begin{equation}\label{quatprop3}
\abs{\inn{Bv,Qv}}\le \sqrt{1-\delta^2}\norm{Q}\abs{Bv},
\end{equation}
which in turn yields
\begin{equation}\label{quatprop2}
\begin{split}
\abs{Av}^2&=\abs{Bv+Qv}^2 = \abs{Bv}^2+\norm{Q}^2+2\inn{Bv,Qv}\\
&\ge (1-\sqrt{1-\delta^2})(\abs{Bv}^2+\norm{Q}^2)+\sqrt{1-\delta^2}(\abs{Bv}-\norm{Q})^2\\
&\ge (1-\sqrt{1-\delta^2})(\abs{Bv}^2+\norm{Q}^2).
\end{split}
\end{equation}
In particular, $A$ is invertible. We also have the trivial estimate
\begin{equation}\label{quatprop1}
\abs{Av}^2\le 2(\norm{Bv}^2+\norm{Q}^2).
\end{equation}
Combining~\eqref{quatprop2} and ~\eqref{quatprop1}, we conclude that
\[\begin{split}
\norm{A}^2\norm{A^{-1}}^2&=
\frac{\max\{\abs{Av}^2\colon \abs{v}=1\}}{\min\{\abs{Av}^2\colon \abs{v}=1\}}\\
&\le \frac{2}{1-\sqrt{1-\delta^2}}
\frac{\max\{\abs{Bv}^2\colon \abs{v}=1\}+\norm{Q}^2}{\min\{\abs{Bv}^2\colon \abs{v}=1\}+\norm{Q}^2} \\
&\le \frac{2}{1-\sqrt{1-\delta^2}}
\frac{\max\{\abs{Bv}^2\colon \abs{v}=1\}}{\min\{\abs{Bv}^2\colon \abs{v}=1\}}\le \frac{2(1+ \sqrt{1-\delta^2})^2}{(1-\sqrt{1-\delta^2})^3}
\end{split}\]
where the last step uses Proposition~\ref{deltaqc}.  This proves~\eqref{quatp1}.
Applying~\eqref{quatp1} to the derivative matrix $A=Df(x)$, we find that $f$ is $K$-quasiconformal with $K= \widetilde{H}(\delta)^3$.

With $A=Df(x)$ and $B=A-\im \HH(A)$ as above, we have
\[\inn{Av,v} = \inn{Bv,v} \ge 0,\quad v\in\R^4.\]
Integrating this inequality along the segments $[a,b]$ on which $f$ is absolutely continuous, we obtain $\Delta_f(a,b)\ge 0$.
The continuity of $f$ then implies $\Delta_f(a,b)\ge 0$ for all $a,b\in\R^4$.
\end{proof}

\begin{remark}\label{remarkk}
If in Definition~\ref{quatred} we do not require $f$ to be homeomorphic, then the proof of Proposition~\ref{quatprop} shows that $f$ is $K$-quasiregular
(see Definition~\ref{qc}).
\end{remark}

It follows from Definition~\ref{quatred} that the set of reduced quasiconformal mappings is a convex cone in four dimensions
as well as in two dimensions. Another similarity with the planar case is provided by the following result.

\begin{proposition}\label{delquat} Any nonconstant $\delta$-monotone mapping $f\colon\R^4\to\R^4$ is
reduced quasiconformal in the sense of Definition~\ref{quatred}.
\end{proposition}
\begin{proof} Since $f$ is quasiconformal by~\cite[Cor. 7]{Ko}, we have $f\in W^{1,4}_{\loc}(\Omega ; \R^4)$.
Fix a point $x\in \R^4$ where $f$ is differentiable and let $A=Df(x)$, $Q=\im \HH(A)$, $B=A-Q$.
The definition of $\delta$-monotonicity implies   $A\in \M_4(\delta)$.
For all $v\in\R^n$ we have $\inn{Qv,v}=0$ since $Q$ is an antisymmetric matrix. Thus,
\begin{equation}\label{shouldname}
\inn{Bv,v}=\inn{Av,v}\ge \delta\abs{Av}\abs{v},\quad v\in\R^n.
\end{equation}
 It remains to prove that $\abs{Av}\ge c\abs{Bv}$ for a constant $c>0$ that depends only on $\delta$.
Note that $\im \HH(A)$ is the orthogonal projection of $A$ onto the space of purely imaginary quaternions, considered as a linear subspace of $\M_4$. Therefore, $B=A-\im\HH(A)$ is the projection of $A$ onto the orthogonal
complement of purely imaginary quaternions. Since orthogonal projections in $\M_n$ do not
increase the Frobenius norm $\norm{\cdot}_F$, it follows that
\begin{equation}\label{aname}
\norm{B}\le \norm{B}_F\le \norm{A}_F\le 2\norm{A}
\end{equation}
where we have used the relation between operator norm and Frobenius norm~\cite[p.313]{HJbook}. Combining
Proposition~\ref{deltaqc} with~\eqref{aname} we obtain
\[\abs{Av} \ge \frac{\norm{A} \abs{v} }{H(\delta)} \ge \frac{\norm{B} \abs{v}}{2 H(\delta)} \ge \frac{\abs{Bv}}{2H(\delta)}. \]
This estimate together with~\eqref{shouldname} imply $B \in \M_4(\delta/2H(\delta))$.
\end{proof}

Our last result is an extension of Theorem~\ref{fvar} to four dimensions.

\begin{theorem}\label{quatvar} Let $f\colon \R^4 \to \R^4$ be a reduced quasiconformal mapping in the sense of Definition~\ref{quatred}.
Then for any $q>1$ $f$ has finite $\phi$-variation on lines with $\phi$ as in~\eqref{fvar1}.
\end{theorem}
\begin{proof}  For a purely imaginary quaternion $Q\in\M_4$ we define $f^Q(x)=f(x)+Qx$.
Recall the definition of the modulus of monotonicity $\Delta_f$ in~\eqref{Deltadef}.
We claim that
\begin{equation}\label{quatvar1}
\Delta_f(a,b)=\min_{Q}\abs{f^Q(a)-f^Q(b)}
\end{equation}
where the minimum is taken over all purely imaginary quaternions (and is attained). Indeed,
\[\Delta_f(a,b)=\Delta_{f^Q}(a,b)\le \abs{f^Q(a)-f^Q(b)},\quad \text{for any $Q$ with }\re Q=0.\]
In proving the converse inequality we may assume that $v:=a-b$ is a nonzero vector.
Applying the unit quaternions $\mathbf{i}$, $\mathbf{j}$, and $\mathbf{k}$ to $v$,
we obtain an orthogonal basis of $\R^4$, namely $\{v,\mathbf{i}v,\mathbf{j}v,\mathbf{k}v\}$.
Expand the vector $f(a)-f(b)$ in this basis:
\begin{equation}\label{quatexp}
f(a)-f(b)=\alpha v+\beta\mathbf{i}v+\gamma\mathbf{j}v+\zeta\mathbf{k}v.
\end{equation}
In these terms, $\Delta_f(a,b)=\inn{\alpha v,v/\abs{v}}=\alpha\abs{v}$. Since $f$ is monotone by Proposition~\ref{quatprop}, we have $\alpha\ge 0$.
The quaternion $Q=-\beta\mathbf{i}-\gamma\mathbf{j}-\zeta\mathbf{k}$ satisfies
\[\abs{f^Q(a)-f^Q(b)}=\alpha\abs{v}=\Delta_f(a,b)\]
which proves~\eqref{quatvar1}.
For future reference, observe that~\eqref{quatexp} implies $\abs{f(a)-f(b)}\ge \abs{Qv}=\norm{Q}\abs{v}$, hence
\begin{equation}\label{quatest}
\norm{Q}\le \frac{\abs{f(a)-f(b)}}{\abs{a-b}}.
\end{equation}

Since linear mappings trivially satisfy the conclusion of the theorem, we may assume that $f$
is nonlinear.   By Remark~\ref{remarkk} there exist $K<\infty$ such that   $f^Q$ is $K$-quasiregular for all purely imaginary quaternions. Also, $f^Q$ is monotone by Proposition~\ref{quatprop}. By Theorem 1.2~\cite{KO} any monotone quasiregular mapping defined on $\R^n$ is either constant or a homeomorphism.
Since $f$ is not linear, $f^Q$ cannot be constant. Thus $f^Q$ is $K$-quasiconformal.
By~\cite[Thm 11.14]{Hebook} the family $f^Q$ has  a common modulus of quasisymmetry $\eta$. Given distinct points $a,b,c\in \R^4$, let $Q$ be a minimizing quaternion in~\eqref{quatvar1}.
The quasisymmetry of $f^{Q}$ implies
\[
\abs{f^{Q}(c)-f^{Q}(a)}\le \eta\left(\frac{\abs{c-a}}{\abs{b-a}}\right)\abs{f^{Q}(b)-f^{Q}(a)}.
\]
Here $\abs{f^{Q}(b)-f^{Q}(a)}=\Delta_f(a,b)$. Using~\eqref{quatest} we obtain
\[\begin{split}
\abs{f(c)-f(a)}&\le \abs{Q(c-a)}+\abs{f^{Q}(c)-f^{Q}(a)} \\
&\le \frac{\abs{c-a}}{\abs{a-b}}\abs{f(a)-f(b)}+\eta\left(\frac{\abs{c-a}}{\abs{b-a}}\right)\Delta_f(a,b),
\end{split}\]
which is~\eqref{fvgen1}. It remains to apply Theorem~\ref{fvgen} to $f$.
\end{proof}

\section*{Acknowledgement}
We thank Tadeusz Iwaniec   for many stimulating discussions on the subject of this paper.

\bibliographystyle{amsplain}

\end{document}